\documentclass[a4paper,twoside,11pt]{article}
\usepackage[english]{babel}
\usepackage[T1]{fontenc}
\usepackage[ansinew]{inputenc}
\usepackage{geometry}
\usepackage{color} 
\usepackage{lmodern}

\usepackage{graphicx}

\usepackage{amsmath} \numberwithin{equation}{section}
\usepackage{amsthm}
\usepackage{amsfonts}
\usepackage{amssymb}
\usepackage{graphics}
\usepackage{color}

\usepackage[usenames,dvipsnames,svgnames,table]{xcolor}
\usepackage[normalem]{ulem}

\theoremstyle{plain}
\newtheorem{theorem}{Theorem}[section]
\newtheorem{lemma}[theorem]{Lemma}

\newtheorem{question}[theorem]{Question} 
\newtheorem{remark}[theorem]{Remark}
\theoremstyle{definition}

\theoremstyle{remark}
\newtheorem{example}[theorem]{Example}

\definecolor{orange}{rgb}{1,0.5,0}

\usepackage{enumerate}

\newcommand\C{\mathbb C}         
\newcommand\R{\mathbb R}         
         
\newcommand\N{\mathbb N}         
\newcommand\Ha{\mathbb H}

\renewcommand\Im{\text{Im}}

      \newcommand{\eps}{\varepsilon}

\newcommand{\LandauO}{\mathcal{O}}

\newcommand{\dist}{\operatorname{dist}}
\newcommand{\B}{\mathbb B}

\newcommand{\supp}{\operatorname{supp}} 

\usepackage[pdfstartview=FitH]{hyperref}

\begin{document}
\parindent 0pt 

\setcounter{section}{0}

\title{Multiple SLE and the complex Burgers equation}
\author{Andrea del Monaco  \and
        Sebastian Schlei\ss inger\thanks{Supported by the ERC grant  ``HEVO - Holomorphic Evolution Equations'' n. 277691.}}
\date{\today}

\maketitle

\abstract{In this paper we ask whether one can take the limit of multiple SLE as the number of slits goes to infinity. In the special case of $n$ slits that connect $n$ points of the boundary to one fixed point, one can take the limit of the Loewner equation that describes the growth of those slits in a simultaneous way. In this case, the limit is a deterministic Loewner equation whose vector field is determined by a complex Burgers equation.}\\



\tableofcontents
\parindent 0pt

\section{Introduction}
The stochastic Loewner evolution (SLE), introduced by O. Schramm in 2000, provides a powerful model to describe certain two dimensional random curves that arise in different contexts in probability theory as well as in statistical physics. \\
For example, SLE can be used to describe the scaling limit of an interface curve of the critical Ising model with a certain boundary condition. A slightly different boundary condition produces several, pairwise disjoint interface curves and so it is a natural question to ask for a generalization of SLE to the case of $n\in\N$ random curves.
Several authors have discussed this generalization of SLE to \emph{multiple SLE}; see \cite{MR2004294},  \cite{MR2187598}, \cite{MR2358649}, \cite{Graham2007}, \cite{MR2310306}. An application to the critical Ising model can be found in \cite{MR2515494}.\\
 In this paper we touch the question what happens if $n\to\infty.$\\

To begin with, we fix some notations. We agree that $\Ha$ will denote the upper half-plane of the complex plane, that is $\Ha=\{z\in\C\,|\, \Im(z)>0\}$, and that $\kappa\in(0,4]$ will be a fixed parameter. Moreover, we let $\{N_l\}_{l\in\N}$ be a sequence of strictly increasing natural numbers.\\

For any $l\in\N$, we assume that there exist $N_l$ points $x_{1,l}<\ldots<x_{N_l,l}$ of $\partial\Ha\setminus\{\infty\}=\R$ such that the set $I_l:=\{x_{i,l}\,|\, i=1,...,N_l\}$ is bounded from either sides by constants that are independent of $l$. Thus we assume
\begin{equation}\label{inan}
\bigcup_{l\in\N} I_l \, \subset \, [-M,M]\, \quad \text{for some $M>0$.}
\end{equation}

Consider the set $\Delta(x_{1,l},\ldots,x_{N_l,l})$ of all $N_l$-tuples of curves $(\gamma_1,\ldots,\gamma_{N_l})$ such that $\gamma_k$ connects $x_{k,l}$ to $\infty$ through $\Ha$ and $\gamma_k \cap \gamma_m \cap \Ha=\emptyset$ whenever $k\not=m$. For each $l\in \N$ the theory of multiple SLE gives us a probability measure $\mu_{\Ha,\kappa}((x_{1,l},...,x_{N_l,l}),\infty)$ that is supported on $\Delta(x_{1,l},\ldots,x_{N_l,l})$. We describe this probability measure in more detail in Section \ref{multipleSLE}.\\
Now we can make sense of the limit $\lim_{l\to\infty}\mu_{\Ha,\kappa}((x_{1,l},...,x_{N_l,l}),\infty)$ in the following way:\\

The deterministic theory of multi-slits evolution allows us to describe the growth of any element of $\Delta(x_{1,l},\ldots,x_{N_l,l})$ by a Loewner equation with ``constant simultaneous growth'':

If $(\gamma_1,...,\gamma_{N_l})\in \Delta(x_{1,l},\ldots,x_{N_l,l}),$ then there exist parametrizations $\Gamma_1(t),\ldots,\Gamma_{N_l}(t)$ for the curves $\gamma_1,...,\gamma_{N_l}$ and a unique conformal mapping $g_t$ from $\Ha\setminus \bigcup_{k=1}^n \Gamma_k[0,t]$ onto $\Ha$ with Laurent expansion at $\infty$ given by
\begin{equation*}
g_t(z) = z+\frac{2t}{z}+\LandauO(|z|^{-2}), \quad z\to\infty, 
\end{equation*}
such that
\begin{equation}\label{multi}\dot{g}_t(z) = \sum_{k=1}^n\frac{2/n}{g_t(z)-V_k(t)}, \quad g_0(z)=z\in\Ha,\end{equation}
where the \emph{driving functions} $V_1,..., V_n$ are uniquely determined, continuous real-valued functions.\\

Now, if we describe the growth of the random curves $\gamma_1,\ldots,\gamma_{N_l}$ from $\mu_{\Ha,\kappa}((x_{1,l},...,x_{N_l,l}),\infty)$ by equation \eqref{multi}, then we can ask for the limit of the process, i.e. for fixed $t\geq0$ we consider the limit
$\lim_{l\to\infty}g^l_t.$\\

We need one further notation:\\
Let $\delta_{x_{k,l}}$ be the Dirac measure centered at $x_{k,l}$ and let $\mu^l$ be the probability measure defined as
\begin{equation}
\mu^l=\frac1{N_l}\sum_{k=1}^{N_l}\delta_{x_{l,k}}
\end{equation}
for any $l\in\N$. Namely, we are assigning to each point $x_{l,k}$ the mass $\frac1{N_l}$ and sum up the point measures.

\begin{theorem}\label{theorem1} Assume that there exists a probability measure $\mu$ such that%
\begin{equation} 
\mu^l \to  \mu \quad \text{weakly as} \quad l\to\infty \,.
\end{equation}
Then $g^l_t$ converges in distribution with respect to locally uniform convergence to the solution $g_t:D_t\to\Ha$ of the deterministic Loewner equation%
\begin{equation}\label{Loewner2}
\dot{g}_t(z) \, = \, M_t(g_t(z)) \,, \quad g_0(z)=z,
\end{equation}%
where $M_t(z)$ is given by the complex Burgers equation%
\begin{equation}\label{burgers2}
\frac{\partial M_t(z)}{\partial t}=-2 M_t(z)\frac{\partial M_t(z)}{\partial z}, \quad M_0(z)=\int_\R\frac{2\mu(du)}{z-u},
\end{equation} %
and it is such that
\begin{equation}
\label{limitbehaviour}\text{$M_t(z)\to 0$ locally uniformly on $\Ha$ for $t\to\infty$.}
\end{equation}

\end{theorem}

Furthermore,

\begin{theorem}\label{form1}
The set $K_t=\Ha\setminus D_t$ is bounded for every $t\geq 0$ and there exists $T>0$ such that for every $t> T$, the boundary $\partial K_t\cap\Ha$ is an analytic curve in $\Ha$.
\end{theorem}

We postpone the proofs of the above results to Section \ref{proof}. In Section \ref{multipleSLE} we will recall the definition of multiple SLE as it was introduced in \cite{MR2310306}, whereas in Section \ref{outlook} we will discuss the example $\mu=\delta_0$ and some further questions.

\section{Multiple SLE}\label{multipleSLE}
In what follows, $\kappa$ is a fixed parameter in $(0,4]$ and $D$ is a Jordan domain of the complex plane $\C$.

\subsection{One-slit SLE}
Fix two points $x,y \in \partial D$ and assume that $\partial D$ is analytic in neighbourhoods of $x$ and $y$.\\
The chordal stochastic Loewner evolution for the data $D,x,y,\kappa$ can be viewed as a certain probability measure $\mu_{D,\kappa}(x,y)$ on the space of all simple curves connecting $x$ to $y$ within $D.$ As one property of SLE is conformal invariance, it suffices to describe SLE when
\begin{equation}
D=\Ha,\quad x=0, \quad \text{and} \quad y=\infty\,.
\end{equation}
In such a case, a random curve $\gamma\in\Delta(0)$ can be efficiently described as follows.\\
Assume $\gamma(t)$ is parametrized by half plane capacity $2t$, i.e. $\gamma(0)=0$ and the conformal mapping $g_t$ from $\Ha\setminus \gamma[0,t]$ onto $\Ha$ with $g_t(z)-z\to0$ for $z\to\infty$ has the expansion 

$$ g_t(z) = z+\frac{2t}{z}+\LandauO(|z|^{-2}) \quad \text{for}\quad z\to\infty.$$

Then $g_t$ satisfies the Loewner equation
\begin{equation}\label{loewner1}
\frac{dg_t(z)}{dt} = \frac{2}{g_t(z) - \sqrt{\kappa}B_t}, \quad g_0(z)=z,
\end{equation}
where $B_t$ is a standard one-dimensional Brownian motion. Of course, one may also consider SLE for $\kappa>4$. But then the measure is no longer supported on $\Delta(0)$ and we are not interested in such a case here.

\begin{remark} For fixed $z\in\Ha$, the solution to  Loewner equations such as \eqref{loewner1} and \eqref{multi} may have a finite lifetime $T(z)>0$ in the sense that $g_t(z)\in\Ha$ for all $t<T(z)$ but $\lim_{t\to T(z)}\Im(g_t(z))=0.$ \\
If we fix a time $t>0$ and let $K_t=\{z\in\Ha \,|\, T(z)\leq t\},$ then $g_t(z)$ maps the domain $\Ha\setminus K_t$ conformally onto $\Ha.$
For more information on the Loewner equation and on SLE, see~\cite{Lawler:2005}.
\end{remark}

\subsection{Multiple SLE}\label{sec-multiple}
In the following we describe multiple SLE as it was introduced in \cite{MR2310306}.\\

Let $N\in\N$ and fix $2N$ points $p_1,...,p_{2N}\in \partial D$ in counter-clockwise order. Assume that $\partial D$ is analytic in a neighbourhood of each $p_k.$\\
We call the pair $(\mathbf{x}, \mathbf{y})$ of two vectors $\mathbf{x}=(x_1,...,x_N),\mathbf{y}=(y_1,...,y_N)$ a \emph{configuration} for these points if 
\begin{itemize}
\item[a)] $\{x_1,...,x_N,y_1,...,y_N\}=\{p_1,...,p_{2N}\}$,
\item[b)] there exist $N$ pairwise disjoint curves $\gamma_1,...,\gamma_n$ in
$D$ such that $\gamma_k$ connects $x_k$ with $y_k,$
\item[c)]  $x_1=p_1$ and $x_1,x_2,...,x_N,$ as well as $x_1,x_k,y_k$, for every $k\geq 2,$ are in counter-clockwise order.

\end{itemize}
The points in $\mathbf{x}$ can be thought of as starting points of these curves. Then $\mathbf{y}$ represents the end points and the assumption in c) just prevents us from getting a new configuration by exchanging a starting point of one curve with its endpoint.
A simple combinatorial exercise shows that there exist 
$$C_N = \frac{(2N)!}{(N+1)!\,N!}$$

many configurations for $2N$ points.\\
Now fix a configuration $(\mathbf{x},\mathbf{y})$. \\
\emph{Configurational multiple SLE} $Q_{D,\kappa}(\mathbf{x},\mathbf{y})$ is a positive, finite measure on the space of all $N-$tuples $(\gamma_1,...,\gamma_N)$ where $\gamma_k$ is a simple curve in $D$ connecting $x_k$ and $y_k$ and $\gamma_k\cap\gamma_j=\emptyset$ whenever $j\not=k.$\\
If we let $H_{D,\kappa}(\mathbf{x},\mathbf{y})$ be the mass of $Q_{D,\kappa}(\mathbf{x},\mathbf{y}),$ then we can write $$Q_{D,\kappa}(\mathbf{x},\mathbf{y})=H_{D,\kappa}(\mathbf{x},\mathbf{y})\cdot \mu_{D,\kappa}(\mathbf{x},\mathbf{y})$$ where $\mu_{D,\kappa}(\mathbf{x},\mathbf{y})$ is some probability measure.\\
The measure $Q_{D,\kappa}(\mathbf{x},\mathbf{y})$ has the following four fundamental properties (see \cite[Section 3.2]{MR2310306} for its construction):
\begin{itemize}
\item[(a)] Conformal covariance: If $f:D\to E$ is a conformal mapping such that $\partial E$ is analytic in a neighbourhood of $f(p_k)$ for every $k,$ then 
$$ f\circ Q_{D,\kappa}(\mathbf{x}, \mathbf{y}) = |f'(\mathbf{x})|^b|f'(\mathbf{y})|^b Q_{E,\kappa}(f(\mathbf{x}, \mathbf{y})), \quad \text{with $b=\frac{6-\kappa}{2\kappa}$}.$$
\item[(b)] The case $N=1:$ $\mu_{\Ha,\kappa}(0,\infty)$ is the chordal SLE($\kappa$) probability measure and $H_{\Ha, \kappa}(0,\infty)=1.$
\item[(c)] Cascade relation (\cite[Proposition 3.2]{MR2310306})
\item[(d)] Boundary Perturbation (\cite[Proposition 3.3]{MR2310306})
\end{itemize}

So $Q_{D,\kappa}(\mathbf{x},\mathbf{y})$ can be thought of as a probability measure with a weight for the underlying configuration. These weights serve as partition functions to combine multiple SLE for different configurations.\\
Indeed, if $S:=\{(\mathbf{x}_1, \mathbf{y_1}), ..., (\mathbf{x}_l, \mathbf{y_l})\}$ is a set of $l$ configurations, then we can consider the new measure
\begin{equation}
Q_{D,\kappa}(S):=\sum_{k=1}^l Q_{D,\kappa}(\mathbf{x}_k,\mathbf{y}_k)=H_{D,\kappa}(S) \cdot \mu_{D,\kappa}(S),
\end{equation} 
where $H_{D,\kappa}(S)$ denotes again the mass of $Q_{D,\kappa}(S)$ and $\mu_{D,\kappa}(S)$ is a probability measure.
In the case $l=C_N$, we consider all possible configurations.
 \begin{example} Consider the case $N=2$ and $\kappa=3$. Then there are two possible configurations $C_1$ and $C_2$, and $\mu_{D,3}(\{C_1,C_2\})$ describes the scaling limit for the Ising model with corresponding boundary conditions (see~\cite{MR2515494}). The probability $p$ for obtaining configuration $C_1$ is given by $$p=\frac{H_{D,3}(C_1)}{H_{D,3}(C_1)+H_{D,3}(C_2)}.$$ \hfill $\bigstar$
\end{example}

Because of conformal invariance, it suffices again to consider the case $D=\Ha$ only, where $p_1,...,p_{2N}\in \R\cup{\{\infty\}}.$
The number $H_{\Ha, \kappa}(\mathbf{x}, \mathbf{y})$ is known explicitly only for some special cases:
\begin{itemize}
\item[(i)] For $N=1,$ we obtain from property (b) by using a  M\"obius transformation: $H_{\Ha,\kappa}(x,y)=|y-x|^{-2b}.$
\item[(ii)] $\kappa=2:$ $H_{\Ha, \kappa}(\mathbf{x}, \mathbf{y}) = |\det[(y_k-x_j)^{-2}]_{j,k}|$ (see~\cite{MR2310306}, the Remark after Proposition 3.3).
\item[(iii)] It can be expressed by a formula involving the hypergeometric function for $N=2,$ (see~\cite[Proposition 3.4]{MR2310306}).
\end{itemize}

Finally we notice that one may consider $Q_{\Ha,\kappa}(\mathbf{x}, \mathbf{y})$ also for a configuration where $y_j=y_k$ (or $x_j=x_k$,  
or both) for certain $j\not=k$. This is done by considering the disjoint case $y_j\not=y_k$ first and then taking a scaled limit.
We include the following case as a definition and refer to \cite[Section 4.6]{MR2187598}, and the references therein.

 \begin{equation}\label{infinity1} H_{\Ha,\kappa}((x_1,...,x_N), \infty) := H_{\Ha,\kappa}((x_1,...,x_N), (\infty,...,\infty)) :=\displaystyle \prod_{1\leq j<k\leq N}(x_k-x_j)^{2/\kappa}.
\end{equation}

\subsection{Can we take the limit?}

Let $N_1<N_2<...$ be a sequence of increasing natural numbers. For each $l\in\N$, pick $2N_l$ points $x_{l,1}<...<x_{l,2N_l}$ on $\R$ and let $(\mathbf{x}_l,\mathbf{y}_l)$ be a configuration.
Now we can ask for a description of the limit $\lim_{l\to\infty}\mu_{D,\kappa}(\mathbf{x}_l,\mathbf{y}_l)$. \\
More generally, following the discussion in Section~\ref{sec-multiple}, we can consider the set $S_l$ of configurations for the points $x_{l,1},...,x_{l,2N_l}$.
\begin{question}\label{q2} Under which conditions and in which sense does the limit
$$\lim_{l\to\infty}\mu_{D,\kappa}(S_l)$$
exist and how can it be described? 
\end{question} 

\begin{remark}
Assume that $S_1,S_2,\ldots$ is a sequence of configurations as above. The measure $\mu_{\Ha, \kappa}(S_l)$ induces a probability measure $\nu_l$ on the finite space $S_l$. So, if we forget about the curves and only think of the configurations, we are lead to the question whether there are there interesting, non-trivial limits of $\nu_l$ for $l\to\infty$.   
\end{remark}

In what follows, we only consider the special case of $N_l$ curves connecting $N_l$ points on the real axis to $\infty$.

\subsection{Simultaneous growth}

Let $N\in\N$ and $x_1<\ldots<x_N$ be $N$ points on $\R$. Furthermore, choose $\lambda_1,\ldots,\lambda_N\in(0,1)$ such that $\sum_{k=1}^n\lambda_k=1$.
The $N$ random curves described by $\mu_{\Ha,\kappa}((x_1,...,x_N),\infty)$ can be generated by the Loewner equation
\begin{equation}\label{multi2}
\dot{g}_t(z) = \sum_{k=1}^N\frac{2\lambda_k}{g_t(z)-V_k(t)}, \quad g_0(z)=z\in\Ha,
\end{equation}
where $\lambda_1,\ldots,\lambda_N>0$ and $\sum_{k=1}^N\lambda_k=1$.

The random driving functions $V_1\ldots,V_N$ are given as the solution of the SDE system
\begin{equation}
dV_k =\kappa \lambda_k \cdot \frac{(\frac{\partial}{\partial x_k}H_{\Ha,\kappa})((V_1,...,V_n),\infty)}{H_{\Ha,\kappa}((V_1,...,V_n),\infty)} dt+\sum_{j\not=k}\frac{2\lambda_j}{V_k-V_j}dt+\sqrt{\kappa \lambda_k}dB_k,
\end{equation} 
where $B_1,\ldots,B_N$ are $N$ independent standard Brownian motions (see~\cite[p.1130]{MR2187598}). From \eqref{infinity1} we obtain
\begin{equation}\label{sigma} dV_k = \sum_{j\not=k}\frac{2(\lambda_k+\lambda_j)}{V_k-V_j}dt+\sqrt{\kappa \lambda_k}dB_k.\end{equation}

\begin{remark}
In  fact, we can also consider the case $\lambda_k\in[0,1].$ Then we describe the corresponding marginal distribution of those curves for which $\lambda_k\not=0.$\\
For instance, consider the case $\lambda_1=1$ and $\lambda_k=0$ for $k\geq 2.$ Then \eqref{multi2} describes only one curve with $dV_1 = \sum_{j\not=1}\frac{2}{V_1-V_j}dt+\sqrt{\kappa}dB_1$ and $dV_k=\frac{2}{V_k-V_1}dt,$ i.e.
$V_k(t) = g_t(x_k)$ for $k\geq 2$ (see \cite[Section 4.2]{MR2310306}). This process is a special SLE($\kappa, \rho$) process (see~\cite[p.1796]{MR2358649}). 
\end{remark} 

\section{Proof of Theorem \ref{theorem1}}\label{proof}

\subsection{The McKean-Vlasov equation and the complex Burgers equation}

We recall that $N_1<N_2<\ldots$ is a sequence of natural numbers and that for every $l\in\N,$ $x_{l,1}<\ldots<x_{l,N_l}$ are $N_l$ points on $\R$ such \eqref{inan} holds.\\
Moreover, for every $l\in \N$, we describe $\mu_{D,\kappa}((x_{l,1},\ldots,x_{l,N_l}),\infty)$ by equation \eqref{multi2} with $\lambda_k=\frac1{N_l}$ for each $k$, i.e.
\begin{equation}\label{Loewner1} \dot{g}^l_t(z) = \frac1{N_l} \sum_{k=1}^{N_l} \frac{2}{g^l_t(z)-V_{l,k}(t)}, \quad g_0(z)=z\in\Ha,\end{equation}
with
\begin{equation}\label{drift1} dV_{l,k} = \frac1{N_l}\sum_{j\not=k}\frac{4}{V_{l,k}(t)-V_{l,j}(t)} dt + 
\sqrt{\frac{\kappa}{N_l}}dB_k(t), \; V_{l,k}(0)=x_{l,k}, \end{equation}
for every $k=1,...,N_l$. We define also $\mu^l_t=\frac1{N_l}\sum_{k=1}^{N_l}\delta_{V_{l,k}(t)}$.\\
We expect that we can define $\mu_t$ as the limit of $\mu^l_t$ for $t\to\infty$ and that equation \eqref{drift1} transforms to a differential equation for $\mu_t$. This is true indeed as it was shown in \cite{MR1217451}. 

\begin{theorem}[\cite{MR1217451}, Theorem 1 and equation (11)]
Assume that $\mu^l_0$ converges weakly to $\mu$ in such a way that there exists a $C^\infty$-function $f_0:\R\to [1,\infty)$ with 
$f_0(x)=f_0(-x)$, $f_0(x)\to\infty$ for $x\to\infty$ and
\begin{equation}\label{cond1}
\sup_{l\in\N} \int_\R f_0(x)\mu^l_0(dx)< +\infty.
\end{equation}
Then, for every $t\geq 0,$ the random measure $\mu^l_t$ converges in distribution with respect to weak convergence to the measure $\mu_t$ which is the unique solution of the initial value problem
\begin{equation}\label{mckean}\frac{d}{dt} \left(\int_\R f(x) \mu_t(dx) \right) = 2 \int_{\R}\int_{\R} \frac{f'(x)-f'(y)}{x-y} \mu_s(dx)\mu_s(dy), \quad \mu_0=\mu,\end{equation}
where $f$ runs through the space of all bounded $C^\infty$-functions. If we let $M_t(z)=\int_{\R}\frac{2\mu_t(du)}{z-u},$ $z\in\Ha,$ then $M_t$ solves the following complex Burgers equation
\begin{equation}\label{burgers} \frac{\partial M_t(z)}{\partial t}=-2 M_t(z)\frac{\partial M_t(z)}{\partial z} .\end{equation} 

\end{theorem} 

Some remarks are in order.

\begin{remark} In \cite{MR1217451}, the authors consider a slightly different equation (see equation (7) therein). Setting $\theta=0$ and $\alpha=4$, it gives equation \eqref{drift1} except that $\sqrt{\frac{\kappa}{N_l}}dB_k(t)$ has to be replaced by $\sqrt{\frac{8}{N_l}}dB_k(t).$ However, it can be easily checked that this change has no effect on the limit behaviour.\\
Furthermore, we notice that the minus sign before $\alpha/2$ in Theorem 1 is not correct, compare equation (4) with conditions (6) (see also~\cite[equation (2.12)]{MR1698948}).
\end{remark}

\begin{remark}
The technical condition \eqref{cond1} is satisfied, e.g., when the support of $\mu^l$ is bounded by a constant independent of $l$, which is our assumption \eqref{inan}. Also, note that every probability measure $\mu$ can be approximated by some $\mu^l$ that satisfy \eqref{cond1} (see the remark after Theorem 1 in \cite{MR1217451}).\end{remark}

\begin{remark}
Equation \eqref{mckean} is a special case of a McKean-Vlasov equation which can also be written as   
$$\frac{\partial \mu_t}{\partial t} = -4\cdot  \frac{\partial (\mu_t H(\mu_t))}{\partial x}  , \quad \mu_0=\mu,$$
where $H(\mu_t)$ denotes the Hilbert transform of $\mu_t$ and the equation is understood in a
 distributional sense (see~\cite[p. 392]{MR1698948}).
\end{remark}

Now we come back to the Loewner equation \eqref{Loewner1}. It can be written as 
$$ g^l_t(z) = z + \int_{0}^t \int_\R \frac{2}{g^l_s(z)-u} \, \mu^l_s(du)ds.$$
For each $s\geq0$ the measure $\mu^l_s$ converges in distribution with respect to weak convergence to $\mu_s$. Thus, the measure $\mu^l_s(du)ds,$ $0\leq s \leq t,$ converges in distribution with respect to weak convergence to $\mu_s(du) ds.$\\
This implies that for each $t\geq 0$ the conformal mapping $g^l_t$ converges in distribution with respect to locally uniform convergence to $g_t$, the solution of \eqref{Loewner2} (see Theorem 1.1 in \cite{quantum} which proves this correspondence for the radial Loewner equation).

\subsection{Proof of \texorpdfstring{\eqref{limitbehaviour}}{}}\label{short}

Next we show that $M_t(z)\to 0$ locally uniformly in $\Ha$ as $t\to\infty.$ \\
Let $z_0\in \Ha$ and denote by $z(t)$ the solution to \begin{equation}\label{sL}\dot{z}(t)=2 M_t(z(t)), \quad z(0)=z_0.\end{equation}
A simple calculation shows that $M_t(z(t))$ is constant:
\begin{eqnarray*} \frac{d M_t(z(t))}{d t} &=& \frac{\partial M_t(z(t))}{\partial t} + \frac{\partial M_t(z(t))}{\partial z} \dot{z}(t)\\
&=&  -2M_t(z(t))\frac{\partial M_t(z(t))}{\partial z} + 2  M_t(z(t)) \frac{\partial M_t(z(t))}{\partial z} = 0.
\end{eqnarray*} 
Hence
 \begin{equation}
\label{ele}M_t(z(t)) = M_0(z(0)).
\end{equation}
 Furthermore, $\ddot{z}(t)=0,$ so $z(t)=z_0+2 M_0(z_0) \cdot t.$ This defines $z(t)$ for all $t\in[0,\infty).$\\
 Let $z_1\in \Ha$ be fixed. Then $M_t(z_1)=M_0(z_0(t))$ where $z_0(t)$ is determined by
 \begin{equation}\label{shift}z_1=z_0(t)+2M_0(z_0(t))\cdot t.\end{equation}
Note that $\Im(M_0(z))<0$ for all $z\in \Ha.$ So $\Im(z_0(t))> \Im(z_1)$ 
and
\begin{equation}
|M_0(z)|\leq\int_{\R}\frac{2\mu_0(du)}{|z-u|}\leq  \int_{\R}\frac{2\mu_0(du)}{\Im(z_1)}=\frac{2}{\Im(z_1)}.
\end{equation}
Hence, $M_0(z)$ is bounded on the set of all $z\in \Ha$ with $\Im(z)> \Im(z_1)$.\\

 Now, when $t$ goes to infinity, $|z_0(t)|$ goes to $\infty$ as well. Otherwise, if $z_0(t)$ had a bounded subsequence $z_0(t_n)$, $n\in\N$, then $M_{0}(z_0(t_n))$ would be bounded as well and \eqref{shift} could not hold for all $n\in\N$. Consequently, $$\text{$M_t(z_1)=M_0(z_0(t))=\int_{\R}\frac{2\mu_0(du)}{z_0(t)-u}\to 0$ \quad for \quad $t\rightarrow\infty.$}$$
As $$|M_t(z)|\leq\int_{\R}\frac{2\mu_t(du)}{|z-u|}\leq  \int_{\R}\frac{2\mu_t(du)}{\Im(z)}=\frac{2}{\Im(z)},$$ 
the family $\{M_t\}_{t\geq0}$ is locally bounded. Thus, the Vitali-Porter theorem implies locally uniform convergence of $M_t(z)$ to 0.

\subsection{Proof of Theorem~\texorpdfstring{\ref{form1}}{}}

The proof of \eqref{form1} is divided into several lemmas. First, we prove the boundedness of $K_t$.

\begin{lemma}\label{1} The set $K_t$ is bounded for every $t\geq 0.$
\end{lemma}
\begin{proof}
Let $x_0\in\R$ with $x_0\not\in \supp \mu_0$ and consider the solutions to the real initial value problem \begin{equation*}
\dot{y}(t)=2 M_{t}(y(t)), \quad y(0)=x_0.
\end{equation*}

By the theory of the real (inviscid) Burgers equation (see \cite[p.77, 78]{MR2238098}), they exist locally and the  lifetime $T(x_0)$ of $y(t)$ is finite, for
\begin{equation}\label{eac} 
M_0'(x_0)=\int_\R\frac{-2\mu_0(du)}{(x_0-u)^2}<0.
\end{equation}
This implies that $y(t)$ will hit $\supp(\mu_t)$ at time $t=T(x_0)$, which is given by 
\begin{equation}
\label{endtime}T(x_0) = \frac{-1}{2 M_0'(x_0)}.
\end{equation}

For $T>0$ we can now compute  $S(T):=\sup(\supp \mu_T)$ as follows.\\
Let $[a,b]$ be the smallest interval containing $\supp \mu_0$ and assume $x_0>b.$ Note that $T(x'_0)>T(x_0)$ for any $x'_0>x_0.$ This gives us a one-to-one correspondence between all times $T>0$ and all $x_0>b.$\\
In order to determine $S(T),$ we can first calculate $x_0(T)>b$ according to \eqref{endtime} and then compute $S(T)=y(T)=x_0(T)+2TM_0(x_0(T)).$ Similarly, we can compute $\inf(\supp \mu_T)$ by considering $x_0<a.$
Consequently, the measure $\mu_t$ has bounded support for every $t\geq0$ which implies that the hull $K_t$ is bounded for every $t\geq 0.$
\end{proof}

\begin{lemma}\label{3}There exists a time $T>0$ such that $\supp \mu_t$ is a bounded interval for all $t\geq T.$
\end{lemma}
\begin{proof}
For $x_0\in\R\setminus \supp \mu_0,$ let $T(x_0)$ be defined as in \eqref{endtime}. Denote with $I_0$ the smallest interval containing $\supp \mu_0$ and let $A=I_0\setminus \supp \mu_0.$ Because of \eqref{eac}, the value $M_0'(x_0)$ is bounded from below on $A$ which implies $T:=\sup_{x_0\in A} T(x_0)<\infty.$\\
 Now let $I_T$ be the smallest interval containing $\supp \mu_T.$ We would like to show that $I_T = \supp \mu_T.$\\
So, assume there exists $x\in I_T \setminus \supp_T.$ Let $J$ be the largest open interval with $x\in J$ that is contained in $I_T \setminus \supp \mu_T.$\\
 On the one hand, there exists a time $s<T$ such that $x\in \supp \mu_s$. For $x\in I_0$, this follows from the construction of $T$, whereas for $x\in I_T\setminus I_0$, it follows from the monotonicity properties of the function $x_0\mapsto T(x_0).$\\
On the other hand, we can solve the backward version of \eqref{burgers3} with initial values in $J$, i.e. 
$$ \dot{y}(t)=-2 M_{T-t}(y(t)), \quad y(0)=y_0 \in J, $$
showing that the distance of $x$ to $\supp \mu_{T-t}$ increases when $t$ goes from $0$ to $T$, a contradiction.  
\end{proof}

\begin{lemma}\label{2} Assume that $\supp \mu_0 \subset [a,b].$ Let $x_0\in \R\setminus[a,b]$ and let $x(t)$ be the solution to \eqref{Loewner2} with initial value $x_0\in\R\setminus[a,b]$, i.e.
\begin{equation}\label{Loewner3}\dot{x}(t) = M_t(x(t)), \quad x(0)=x_0.\end{equation}
Then $x(t)$ has a positive finite lifetime, in the sense that there exists $0<S$ such that 
$$\text{$x(t)\not\in\supp \mu_t$ for $t<S$ \quad and \quad $\lim_{t\uparrow S}\dist(x(t),\supp \mu_t) =0.$}$$
\end{lemma}
\begin{proof}
Without of loss of generality, we can consider only the case $x_0>b.$\\
The solution $y(t)$ to \begin{equation}\label{burgers3}\dot{y}(t)=2 M_{t}(y(t)), \quad y(0)=x_0,\end{equation}
 will hit $\supp(\mu_t)$ at $t=T(x_0)$.
Now we compare $x(t)$ with $y(t).$ As $M_0(x_0)=\int_{[a,b]}\frac{2\mu_0(du)}{x_0-u}>0$ we have $0<\dot{x}(0)<2\dot{x}(0)=\dot{y}(0)$ and consequently, $y(t)>x(t)$ for $t$ small enough. \\
Assume that $x(t)$ does not hit $\supp \mu_t$ for $t\in[0,T(x_0)].$ Then there is a first time $t_0\leq T(x_0)$ with $x(t_0)=y(t_0).$ Hence there exists an interval $[t_0-\eps,t_0]$ such that 
$$\dot{x}(t)=M_t(x(t))<2M_t(y(t))=\dot{y}(t)$$ 
for all $t\in[t_0-\eps,t_0].$ As $x(t)\leq y(t)$ in that interval, we cannot have $x(t_0)=y(t_0).$ \\
So $x(t)$ hits $\supp \mu_t$ and stays away from $y(t).$ As a consequence, there exists a time $S< T(x_0)$ such that $x(t)\not\in\supp \mu_t$ for $t<S$ and $\lim_{t\uparrow S}\dist(x(t),\supp \mu_t)=0.$
\end{proof}

Now we complete the proof of Theorem \ref{form1}.
 
\begin{lemma} There exists $T>0$ such that the boundary $\partial K_t\cap\Ha$ is an analytic curve in $\Ha$ for every $t>T$.
\end{lemma}
\begin{proof}
First, consider the Loewner equation 
\begin{equation}\label{simpler} \dot{f}_t(z) = 2M_t(z), \quad f_0(z)=z\in \Ha.
\end{equation}

Let $L_t$ be the generated hull, i.e. $z\mapsto f_t(z)$ is a conformal mapping from $\Ha\setminus L_t$ onto $\Ha.$\\
Fix a time $t_0>0$ and let $z_0\in\Ha$ be a point such that 
$$\text{$f_t(z_0)\in \Ha$ for every $t<t_0$ and $\Im(f_t(z_0))\to0$ for $t\uparrow t_0$.} \qquad (*)$$ 
This condition implies that $z_0$ belongs to $\partial L_{t_0}.$\\
Now, since the function $t\mapsto f_t(z_0)$ is a straight line (see Section \ref{short}) we can extend $f_{t_0}$ analytically to a neighbourhood $U$ of $z_0$, and from $$U\cap \partial L_{t_0} = f_{t_0}^{-1}(f_{t_0}(U)\cap \{\Im(z)=0\}),$$
we see that $\partial L_{t_0}$ is an analytic curve in a neighbourhood of $z_0$. \\
Furthermore, as $f_t(z_0)$ belongs to the lower half-plane for $t>t_0,$ the sets $L_t$ are ``uniformly growing'' in the sense that if $w\in\Ha$ with $w\in\partial L_t$ for some $t\geq 0$, then $w\not\in \partial L_s$ whenever $s\not=t.$ Hence,  condition $(*)$ is in fact equivalent to $z_0 \in \partial L_t$ and, consequently, $\partial L_t \cap \Ha$ is locally an analytic curve.\\

Lemma \ref{3} implies that there exists a time $T>0$ such that $L_t$ is connected for all $t\geq T$, and so $\partial L_t\cap \Ha$ is connected for every $t>T.$ Thus, for every $t>T,$ 
$\partial L_t\cap \Ha$ is an analytic curve that connects two points $a_t$ and $b_t$ on the real axis, with $\supp \mu_0 \subset (a_t,b_t)$.\\
Let now $f_t(a_t)$ and $f_t(b_t)$ denote the continuous extension of $f_t$ to the points $a_t$ and $b_t.$ \\
From \eqref{ele} we know that $M_t(z)=M_0(f_t^{-1}(z))$, and so $M_t(z)$ can be extended analytically in a neighbourhood of  every $x\in(f_t(a_t), f_t(b_t)).$\\

Now we come to the Loewner equation for $g_t,$ namely
$$\dot{g}_t(z) = M_t(g_t(z)), \quad g_0(z)=z\in\Ha.$$
Fix some $t_0> T$ and let $z_0\in\Ha$ be a point such that 
$$\text{$g_t(z_0)\in \Ha$ for every $t<t_0$ and $\Im(g_t(z_0))\to0$ for $t\uparrow t_0$.} \qquad (**)$$
Then $z_0\in \partial K_{t_0}.$ As $t_0> T,$ the support of $\mu_{t_0}$ is the bounded interval $I_{t_0}=[f_t(a_t), f_t(b_t)]$. When $t\uparrow t_0$, $g_t(z)$ approaches $I_{t_0}$. From Lemma \ref{2} we know that the boundary points of this interval correspond to two real values, i.e. $g_t(\hat{a})=f_t(a_t)$ and $g_t(\hat{b})=f_t(b_t)$ for some $\hat{a}<\hat{b}.$ So $g_t(z_0)$ hits the interior of $I_t$ and since $M_t(z)$ can be extended there analytically, we can also extend $g_{t_0}(z_0)$ analytically to a neighbourhood of $z_0.$\\
Analogously to equation \eqref{simpler}, we have $\Im(M_{t_0}(g_{t_0}(z_0)))<0$ and so $g_{t}(z_0)$ belongs to the lower half-plane when $t>t_0$. We conclude that $(**)$ is equivalent to $z_0\in \partial K_{t_0}.$ \\
Consequently, $\partial K_{t_0}\cap \Ha$ is an analytic curve for $t_0>T.$
\end{proof}

\begin{remark} It is worth noting that the solution $g_t$ of the Burgers-Loewner system \eqref{Loewner2}, \eqref{burgers2} can be calculated by solving one ordinary differential equation only.\\
Indeed, let $z\in\Ha$. As we have seen in Section \ref{short}, we can write $M_t(g_t(z))=M_0(z_0(t))$ for some $z_0(t)=:h_t(z)$ that satisfies
\begin{equation}\label{uno} g_t(z) = h_t(z) + 2t M_t(g_t(z)) = h_t(z) + 2tM_0(h_t(z)),\quad h_0(z)=z. \end{equation}
Now we differentiate the last equation with respect to $t$. As $\dot{g}_t(z)=M_t(g_t(z))=M_0(h_t(z))$, we obtain a differential equation for $h_t(z),$ namely
\begin{equation}\label{Andi} \frac{d}{dt}h_t(z) = \frac{-M_0(h_t(z))}{1+2tM_0'(h_t(z))},\quad h_0(z)=z. \end{equation}
Note that the denominator is always $\not=0$ as from \eqref{eac} $\Im(M_0'(z))\not=0$ for all $z\in\Ha$.
\end{remark}

\section{Example and remarks}\label{outlook}

\begin{example}
Assume that $\mu=\delta_0.$ Burgers' equation \eqref{burgers} can be solved explicitly in this case. For $t\geq0$, define
 $$M_t:\Ha\to\C, \quad M_t(z) = \frac{4}{z+\sqrt{z^2-16t}},$$
where we choose the holomorphic branch of the square root such that $\sqrt{-1}=i.$ It can be easily seen that $-M_t$ maps $\Ha$ into $\Ha$ and that
$\lim_{y\to\infty}y\cdot \Im(-M_t(iy))=2$. Thus, $M_t$ has the form 
$$M_t=\int_{\R}\frac{2\mu_t(du)}{z-u},$$
where $\mu_t$ is a probability measure (see \cite[Section 1]{MR1201130}).\\
A simple calculation shows that $M_t$ satisfies \eqref{burgers} and that $M_0(z)=\frac2{z}=\int_\R\frac{2\mu(du)}{z-u}.$ In particular, we obtain
\begin{equation}
\label{supp}\supp(\mu_t) = [-4\sqrt{t}, 4\sqrt{t}].
\end{equation}
Denote with $g_t$ the solution to \eqref{Loewner1}.  It can easily be checked that for every $c>0,$ the family $g_{c^2 t}(c\cdot z)$ also satisfies \eqref{Loewner1} with the same probability measures $\mu_t.$ Thus, if $g_t$ maps the domain $\Ha\setminus K_t$ conformally onto $\Ha,$ then $K_{c^2 t} = c\cdot K_t,$ or $K_{t} = \sqrt{t}\cdot K_1.$ Figure \ref{figure1} shows a numerical approximation of $K_1$.\\
We can also compute $g_t(z)$ explicitly. Equation \eqref{uno} and \eqref{Andi} become
\begin{equation}\label{uno2} g_t(z) = h_t(z) + \frac{4t}{h_t(z)}, \end{equation}
\begin{equation}\label{Andi2} \frac{d}{dt}h_t(z) = \frac{-2/h_t(z)}{1-4t/h_t(z)^2},\quad h_0(z)=z. \end{equation}
The solution to \eqref{Andi2} is given by
\begin{equation}\label{due} h_t(z) = i \cdot \sqrt{\frac{4t}{PL(-4t/z^2)}},\end{equation}
where $PL$ denotes the principal branch of the product logarithm. More precisely: the function $z\mapsto -4t/z^2$ maps $\Ha$ onto $\C\setminus (-\infty, 0].$ $PL$ can be defined in this domain as the branch of the inverse function $w\mapsto w\cdot e^w$ with $PL(e)=1.$ Then $PL$ maps $\C\setminus (-\infty, 0]$ into itself and we chose the square root such that $\sqrt{1}=1,$ i.e. $z\mapsto\sqrt{\frac{4t}{PL(-4t/z^2)}}$ maps $\Ha$ into the right half-plane. Finally, $h_t(z)\in \Ha$ for every $z\in \Ha$ and $t\geq 0.$\\
The value $h_0(z)$ is defined as a limit: $$h_0(z)=\lim_{t\to 0}i \cdot \sqrt{\frac{4t}{PL(-4t/z^2)}}=i \cdot \sqrt{\frac{4}{-4PL'(0)/z^2}}=i\cdot \sqrt{-z^2}=i \cdot z/i = z.$$
Note that $\sqrt{-z^2}=z/i$ for $z\in\Ha$ according to our choice of the square root branch.\\
It can be easily verified that $h_t$ solves \eqref{Andi2}.
Consequently, $g_t(z)$ is given by combining \eqref{uno2} and \eqref{due}. Finally, we can show that
 $$\overline{K_t}\cap \R=[-2\sqrt{et}, 2\sqrt{et}]$$
by recalling \eqref{supp} and verifying $g_t(\pm 2 \sqrt{et})=\pm 4\sqrt{t}.$ 
\hfill $\bigstar$
\end{example}

\begin{figure}[h] \label{figure1}
    \centering
   \includegraphics[width=140mm]{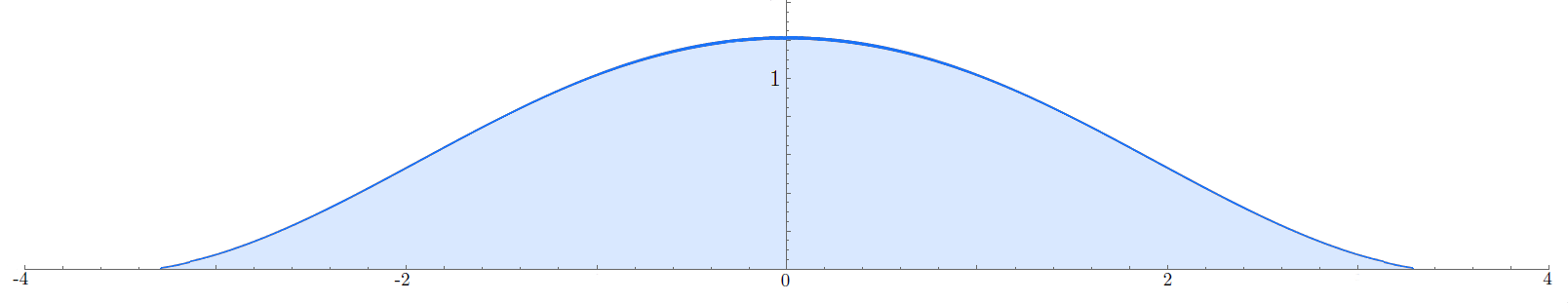}
\caption{The hull $K_1,$ which satisfies $\overline{K_1}\cap\R = [-2\sqrt{e}, 2\sqrt{e}] \approx [-3.3, 3.3].$}
 \end{figure}

\begin{remark} A natural question is to ask for the limit behaviour of multiple SLE described by the more general equation \eqref{multi2}. For each $l\in\N,$ we choose numbers $\lambda_{l,1},...,\lambda_{l,N_l}$ from $(0,1)$ which sum up to $1$. Now define the random measure $\mu_t^l=\sum_{k=1}^{N_l}\lambda_{l,k}\delta_{V^l_k(t)}.$ Then the Loewner equation has the form 
$$\dot{g}_t^l = \int_{\R}\frac{2\mu^t_l(du)}{g_t^l-u}.$$ Under which conditions does the limit for $l\to\infty$ exist?
\end{remark}

\begin{remark} It might be of interest to study variations of \eqref{burgers2} in the context of Loewner theory. In \cite{MR1217451}, e.g., the authors derive the more general equation 
$$\frac{\partial M_t(z)}{\partial t}=(\theta z-2 \frac{\partial M_t(z)}{\partial z})M_t(z)+\theta M_t(z), \quad
M_0(z)=\int_\R\frac{2\mu_0(du)}{z-u},$$
 where, again, $M_t(z)=\int_\R\frac{2\mu_t(du)}{z-u}$ for a probability measure $\mu_t.$
In this case, the limit behaviour of $M_t$ has the following analogue to \eqref{limitbehaviour}: $\mu_t$ converges for $t\to\infty$ to the Wigner semicircle measure whose density is given by $\frac{2}{\pi R^2}\sqrt{R^2-x^2}$, $-R \leq x \leq R,$ with $R=\sqrt{\frac{8}{\theta}};$ see \cite[Section 5]{MR1217451}.
\end{remark}

\begin{remark} Conformal slit mappings of the form $g:\Ha\setminus \gamma\to \Ha,$ where $\gamma$ is a simple curve don't have a straightforward generalization to the higher dimensional setting of biholomorphic mappings on, say, the Euclidean unit ball $\B_n\subset \C^n$, in the sense that $\B_n$ minus a simple curve cannot be mapped onto $\B_n$ biholomorphically for $n\geq 2$. However, the limit equations \eqref{Loewner2} and \eqref{burgers2} can be generalized because of their simple form. As an example, let $\Ha_n$ be the Siegel upper half-space $\Ha_n = \{(z_1,\tilde{z})\in \C^n \,|\, \Im(z_1)> |\tilde{z}|^2\},$ which is biholomorphic equivalent to $\B_n.$\\
Let $-M$ be an infinitesimal generator on $\Ha_n$ in the sense of \cite[Section 1]{MR2887104}. Let $M_t$ be the solution to 
$$ \frac{dM_t(z)}{dt} = -2\cdot DM_t(z)\cdot M_t(z), \quad M_0(z)=M(z), $$
where $DM_t$ denotes the Jacobi matrix of $M_t(z)$ with respect to the $z$-variables. Provided that the solution exists and that $-M_t$ is an infinitesimal generator on $\Ha_n$ for every $t\geq 0,$ then $(z,t)\mapsto M_t(z)$ is a Herglotz vector field and we can consider the Loewner equation $\dot{g}_t=M_t(g_t(z)), g_0(z)=z$ (see again \cite[Section 1]{MR2887104}).
\end{remark}

\providecommand{\bysame}{\leavevmode\hbox to3em{\hrulefill}\thinspace}
\providecommand{\MR}{\relax\ifhmode\unskip\space\fi MR }
\providecommand{\MRhref}[2]{%
  \href{http://www.ams.org/mathscinet-getitem?mr=#1}{#2}
}
\providecommand{\href}[2]{#2}

{\small
\emph{Andrea del Monaco}: 	Universit\`{a} di Roma ``Tor Vergata'', 00133 Roma, Italy.\\
     Email:         delmonac@mat.uniroma2.it  \\
\emph{Sebastian Schlei\ss inger}: Universit\`{a} di Roma ``Tor Vergata'', 00133 Roma, Italy.\\
	Email:	schleiss@mat.uniroma2.it
}

\end{document}